\documentclass[a4paper]{amsart}
\usepackage[latin1]{inputenc}
\usepackage{amssymb}
\usepackage{amsmath}
\usepackage{mathrsfs}
\usepackage{eufrak}
\usepackage{amsthm}
\usepackage{amsfonts}
\usepackage{textcomp}
\usepackage{graphicx}
\usepackage[pdftex]{color}
\usepackage{paralist}
\usepackage[shortlabels]{enumitem}
\usepackage{hyperref}
\usepackage{comment}
\usepackage[arrow, matrix, curve]{xy}

\newtheorem*{corollary*}{Corollary}
\newtheorem*{theorem*}{Theorem}
\newtheorem{theorem}{Theorem}[section]

\newtheorem{corollary}[theorem]{Corollary}
\newtheorem{lemma}[theorem]{Lemma}
\newtheorem{proposition}[theorem]{Proposition}

\newtheorem*{claim*}{Claim}
\newtheorem*{conjecture}{Conjecture}

\theoremstyle{definition}

\theoremstyle{remark}

\numberwithin{equation}{theorem}

\makeatletter
\renewcommand*\env@matrix[1][\
arraystretch]{%
  \edef\arraystretch{#1}%
  \hskip -\arraycolsep
  \let\@ifnextchar\new@ifnextchar
  \array{*\c@MaxMatrixCols c}}
\makeatother

\begin{document}

\title{On the injective dimension of the Jacobson radical}
\date{\today}

\subjclass[2010]{Primary 16G10, 16E10}

\keywords{global dimension, injective dimension, Jacobson radical}

\author{Ren\'{e} Marczinzik}
\address{Institute of algebra and number theory, University of Stuttgart, Pfaffenwaldring 57, 70569 Stuttgart, Germany}
\email{marczire@mathematik.uni-stuttgart.de}

\begin{abstract}
We conjecture that the injective dimension of the Jacobson radical equals the global dimension for Artin algebras. We provide a proof of this conjecture in case the Artin algebra has finite global dimension and in some other cases.
\end{abstract}

\maketitle
\section*{Introduction}
Recall that for a ring $R$ the global dimension $gldim(R)$ of $R$ is defined as the supremum of the projective dimensions of modules.
In \cite{A}, Auslander proved the fundamental result that for semi-primary rings the global dimension equals the maximum of the projective dimensions of the simple modules. For a modern proof of Auslander's result we refer to the book by Lam, see \cite{L} theorem 5.72., where one can also find applications of this result.
Let $J$ denote the Jacobson radical of a ring $R$. Then the result of Auslander can be equivalently stated as $gldim(R)=pd(J)+1$, where $pd(J)$ denotes the projective dimension of the Jacobson radical $J$. Thus in order to calculate the global dimension, it is enough to know the projective dimension of a single module, namely the Jacobson radical of the ring. Suprisingly it seems that no attention has been paid to the injective dimension of the Jacobson radical in the literature yet. In this article we suggest the following conjecture:
\begin{conjecture}
Let $A$ be an Artin algebra. Then the global dimension of $A$ equals the injective dimension of the Jacobson radical $J$ of $A$.
\end{conjecture} 
We prove this conjecture for several classes of algebras in this article. The main theorem proves it for algebras of finite global dimension:
\begin{theorem*}
Let $A$ be an Artin algebra of finite global dimension. Then the global dimension of $A$ equals the injective dimension of the Jacobson radical of $A$.
\end{theorem*}
We further prove the conjecture for some other well studied classes of algebras such as local, selfinjective, Nakayama and Gorenstein algebras.

Corollary \ref{gorensteincorollary} is due to Dan Zacharia.
The author thanks Dan Zacharia for useful discussions and for allowing him to use corollary \ref{gorensteincorollary} in this article.

\section{The injective dimension of the Jacobson radical}
We assume that all algebras are connected non-semisimple Artin algebras and all modules are finitely generated right modules if nothing is stated otherwise.
We assume that the reader is familiar with the basics of Artin algebras as explained for example in the book \cite{ARS} or in \cite{AnFul}. We denote by $pd(M)$ the projective dimension of a module $M$ and by $id(M)$ the injective dimension of $M$.
We denote by $A$ an algebra and by $J$ its Jacobson radical. Since the equation $gldim(A)=id(J)$ is invariant under Morita equivalence, we can assume that the algebra $A$ is basic, that is $A/J$ is isomorphic to the direct product of division rings.
Recall that the \emph{finitistic projective dimension} of an algebra $A$ is defined as the supremum of all projective dimensions of modules with finite projective dimension. 
Dually, the \emph{finitistic injective dimension} of an algebra $A$ is defined as the supremum of all injective dimensions of modules with finite injective dimension. It is a famous open problem whether the finitistic dimension is always finite for Artin algebras. Recall that an algebra is called \emph{QF-3 algebra} in case the injective envelope of the regular module is projective. Famous example of QF-3 algebras are Nakayama algebras (or sometimes called serial algebras in the literature), which are by definition algebras such that every indecomposable module is uniserial. That Nakayama algebras are QF-3 algebras can be found for example in \cite{AnFul} theorem 32.2.

\begin{theorem}
Let $A$ be an Artin algebra. 
\begin{enumerate}
\item The global dimension of $A$ equals the supremum of the injective dimensions of modules
\item The global dimension of $A$ equals the maximum of the projective dimension of the simple $A$-modules.
\item The global dimension of $A$ equals the maximum of the injective dimension of the simple $A$-modules.
\end{enumerate}
\end{theorem}
\begin{proof}
\begin{enumerate}
\item This is well known, see for example \cite{L}, corollary 5.71.
\item This is the classical result of Auslander mentioned in the introduction, see for example \cite{L} theorem 5.72. for a proof.
\item Let $B=A^{op}$ be the opposite algebra of $A$. Then by (1) the global dimension of $B$ equals the maximum of projective dimensions of simple modules of $B$. Let $S$ be a simple $B$-module with projective dimension equal to the global dimension of $B$.
Applying the duality and noting that $A$ and $B$ have the same global dimension, $D(S)$ is a simple $A$-module with injective dimension equal to the global dimension of $A$.
\end{enumerate}
\end{proof}
\begin{lemma} \label{injdimlemma}
Let $0 \rightarrow X \rightarrow Y \rightarrow Z \rightarrow 0$ be a short exact sequence. Then
\begin{enumerate}
\item $pd(X) \leq max(pd(Y),pd(Z)-1)$. 
\item $id(Z) \leq max( id(Y), id(X)-1)$.
\end{enumerate}
\end{lemma}
\begin{proof}
For (1), see \cite{ASS} A.4. proposition 4.7. (b) in the appendix of the book.
(2) is dual to (1).

\end{proof}
\begin{proposition} \label{firstpropo}
Let $A$ be an Artin algebra. Then the global dimension of $A$ equals the injective dimension of $J$ in the following cases:
\begin{enumerate}
\item $A$ is an algebra with finitistic injective dimension equal to zero.
\item $A$ is a Nakayama algebra.
\end{enumerate}
\end{proposition}
\begin{proof}
\begin{enumerate}
\item Recall that we assume that our algebras are non-semisimple. Assume that $A$ has finitistic injective dimension equal to zero. Then $A$ has infinite global dimension, since all non-injective modules must have infinite injective dimension by assumption. Assume $J$ does not have infinite injective dimension and therefore must be injective. Then the short exact sequence
$$0 \rightarrow J \rightarrow A \rightarrow A/J \rightarrow 0$$ 
splits and thus every simple module is a direct summand of $A$ and is therefore projective. This would mean that the global dimension of $A$ is zero as the global dimension equals the maximum of the projective dimensions of the simple modules. This is a contradiction and thus the injective dimension of $J$ must be infinite.
\item By (1), we can assume that $A$ is not selfinjective, as selfinjective algebras have finitistic injective dimension equal to zero. \newline
Let $A$ be a (nonselfinjective) Nakayama algebra with global dimension $g>0$. Let $P$ be an indecomposable projective module with injective dimension g (such a module exists by \cite{ARS} VI. 5. lemma 5.5). Then $P$ is not injective and thus there is an embedding $P \rightarrow I$ where $I$ is the injective envelope of $P$. But $I$ is projective, since Nakayama algebras are QF-3 algebras . Let $I=e_i A$. Then we can write $P=e_i J^k$ for some $k$, since $e_iA$ is a uniserial module. But with $e_i J^k$, all the modules $e_i J^l $ for $l=0,1,2,...k$ are also projective (this follows from the dual of theorem 32.6. of \cite{AnFul}). Thus we can write P as the radical of the projective module $e_i J^{k-1}$. Therefore, there exists a direct summand of the Jacobson radical with injective dimension $g$ and thus the injective dimension of the Jacobson radical is itself equal to $g$.
\end{enumerate}
\end{proof}

\begin{corollary}
Let $A$ be an Artin algebra that is local or selfinjective. Then $gldim(A)=id(J)$.
\end{corollary}
\begin{proof}
This follows immediately from \ref{firstpropo}, since every local or selfinjective algebra has the property that an indecomposable module is either injective or has infinite injective dimension and thus the finitistic injective dimension for such algebras is zero.

\end{proof}

Before the next proposition, we remind the reader on the notion of \emph{quiver of an algebra}, see for example \cite{SY2} for this notion and chapter 11.1. of \cite{HGK} for algebras that are more general than Artin algebras. Let $A$ be a basic Artin algebra with a decomposition of $1_A$ into a sum of pairwise orthogonal primitive idempotents $e_i$:
$$1_A = \bigoplus\limits_{k=1}^{n}{e_i}.$$
The quiver of $A$, denoted by $Q(A)$, is defined as the graph with point $1,...,n$ ($i$ corresponds to the primitive idempotent $e_i$) and an arrow from $i$ to $j$ if and only if $e_i J e_j /e_i J^2 e_j \neq 0$. One can also give each arrow a weight and make $Q(A)$ into a valued quiver, see \cite{SY2} chapter VII., but we do not need this here. Note that $Q(A)$ is a connected graph if and only if $A$ is a connected algebra, see for example \cite{HGK} theorem 11.1.9. Since we assume that all our algebras are connected, also the quivers of our algebras will be connected.
Corresponding to the primitive idempotents $e_i$ (and the point $i$ in the quiver of $A$) there are the pairwise non-isomorphic indecomposable projective modules $e_iA$ with their tops $S_i$, which are the pairwise non-isomorphic simple $A$-modules.
We need the following lemma:
\begin{lemma} \label{simpleinjlemma}
Let $A$ be an Artin algebra and $S$ a simple $A$-module. Then $S$ is injective iff there is no arrow in $Q(A)$ that ends at $i$.
\end{lemma}
\begin{proof}
Let $S=S_i$.
There is an arrow from $j$ to $i$ in $Q(A)$ iff $Ext_A^1(S_j,S_i) \neq 0$ by theorem 1.9. of \cite{SY2}. Now a module $M$ is injective iff $Ext_A^1(S_j,M)=0$ for all simple modules $j=1,...,n$. Thus $S_i$ is injective iff $Ext_A^1(S_j,S_i) =0$ for all $j$ iff there is no arrow from a point $j$ to $i$ in $Q(A)$.
\end{proof}

\begin{proposition} \label{mainpropo}
Let $A$ be an Artin algebra, then the injective dimension of $J/J^2$ equals the global dimension of $A$.

\end{proposition}
\begin{proof}
Note that $J/J^2$ is a direct sum of simple modules, since this module is semisimple. Let $S$ be a simple module that is not injective corresponding to the point $i$ in the quiver $Q(A)$ of $A$. Then there is an arrow starting at a point $j$ to $i$ in $Q(A)$ or else $S$ would be injective by \ref{simpleinjlemma}.
Now $S$ is a direct summand of $J/J^2$ since $e_j J/J^2 e_i \neq 0$ and $J/J^2$ is semisimple. But the global dimension of an algebra equals the supremum of the injective dimensions of simple modules. Thus there is a simple non-injective module with injective dimension equal to the global dimension that is a summand of $J/J^2$, showing that $J/J^2$ has injective dimension equal to the global dimension.

\end{proof}

\begin{theorem} \label{maintheorem}
Let $A$ be an Artin algebra of finite global dimension $g$. Then the injective dimension of the Jacobson radical of $A$ equals the global dimension of $A$.

\end{theorem}
\begin{proof}
Note that we assume that our algebras are not semisimple and thus $g>0$.
We have the short exact sequence 
$$0 \rightarrow J^2 \rightarrow J \rightarrow J/J^2 \rightarrow 0,$$
which gives that $id(J/J^2) \leq max(id(J),id(J^2)-1)$ by \ref{injdimlemma}.
By \ref{mainpropo} we have $id(J/J^2)=g$, and thus the inequality gives
$g \leq max(id(J),id(J^2)-1)$.
Now since we assume that $A$ has finite global dimension, $id(J^2)-1 \leq g-1$ and the inequality $g \leq max(id(J),id(J^2)-1)$ can only hold if $id(J)=g$, which proves the theorem.
\end{proof}

We get two corollaries from the previous theorem. Recall that an algebra $A$ is called \emph{Gorenstein} in case the injective dimension of the left and right regular modules are finite and coincide. In this case the \emph{Gorenstein dimension} of an algebra is by definition the injective dimension of the regular module $A$. The class of Gorenstein algebras contains all algebras of finite global dimension and for algebras of finite global dimension the Gorenstein dimension coincides with the global dimension. Using \ref{maintheorem} we can prove that our conjecture on the injective dimension of the Jacobson radical also holds for Gorenstein algebras. The next corollary and its proof are due to Dan Zacharia.
\begin{corollary} \label{gorensteincorollary}
Let $A$ be a Gorenstein algebra. Then the injective dimension of the Jacobson radical coincides with the Gorenstein dimension of $A$.
\end{corollary}
\begin{proof}
In case $A$ has finite global dimension, the result follows from \ref{maintheorem}. Thus assume $A$ is Gorenstein with infinite global dimension. We have to prove that the Jacobson radical $J$ has infinite injective dimension.
Look at the following short exact sequence, where the injective map is given by the inclusion of the radical into the regular module:
$$0 \rightarrow J \rightarrow A \rightarrow A/J \rightarrow 0.$$
The module $A/J$ is a direct sum of simple modules and contains each simple $A$-module at least once as a direct summand. Since the global dimension equals the maximum of injective dimension of simple modules, $A/J$ has infinite injective dimension. By assumption, $A$ has finite injective dimension. Assume now that $J$ also has finite injective dimension.
Then looking at the above short exact sequence and using \ref{injdimlemma}, we get that $id(A/J) \leq max(id(A),id(J)-1) < \infty$, which is a contradiction. Thus $J$ must have infinite injective dimension.

\end{proof}
We remark that the above proof of the theorem \ref{maintheorem} also shows that the equality $gldim(A)=id(J)$ holds for algebras $A$ where one has $id(J) \geq id(J^2)-1$. We note that as a corollary:
\begin{corollary}
Let $A$ be an Artin algebra with Jacobson radical $J$ such that $id(J) \geq id(J^2)-1$. Then the global dimension of $A$ equals the injective dimension of $J$.
This is especially true for radical square zeros algebras, namely such algebas where $J^2=0$.
\end{corollary}

We remark that we tested the inequality $id(J) \geq id(J^2)-1$ for various classes of algebras with the computer and in all our examples we even had that $id(J) \geq id(J^2)$.

\end{document}